\documentclass[a4paper,12pt]{article}
\usepackage{amsthm,amsmath}
\usepackage{amsfonts,amssymb}
\newtheorem{theorem}{Theorem}
\newtheorem{lemma}[theorem]{Lemma}
\newtheorem*{thm2}{Theorem 2}
\def\lm {\preceq}
\def\F {{\cal F}}
\begin{document}
\title{Graphs with bounded tree-width and large odd-girth are almost bipartite}
\author{Alexandr V. Kostochka\thanks{%
Department of Mathematics, University of Illinois, Urbana, IL 61801
and Institute of Mathematics, Novosibirsk 630090, Russia. E-mail: 
\texttt{kostochk@math.uiuc.edu}. This author's work was partially
supported by NSF grant DMS-0650784 and by grant 09-01-00244-a of the
Russian Foundation for Basic Research.}
\and
Daniel Kr{\'a}l'\thanks{%
Institute for Theoretical Computer Science, Faculty of Mathematics and
Physics, Charles University, Malostransk{\'e} n{\'a}m{\v e}st{\'\i} 25,
118 00 Prague, Czech Republic. E-mail: \texttt{kral@kam.mff.cuni.cz}.
The Institute for Theoretical Computer Science (ITI) is supported by
Ministry of Education of the Czech Republic	as project 1M0545.
This research has also been supported by the grant GACR 201/09/0197.}
\and
Jean-S{\'e}bastien Sereni\thanks{%
CNRS (LIAFA, Universit\'e Denis Diderot), Paris, France, and Department
of Applied Mathematics (KAM), Faculty of Mathematics and Physics,
Charles University, Prague, Czech Republic. E-mail:
\texttt{sereni@kam.mff.cuni.cz}.}
\and
Michael Stiebitz\thanks{%
Technische Universit\"at Ilmenau, Institute of Mathematics,
P.O.B. 100 565, D-98684 Ilmenau, Germany. E-mail:
\texttt{Michael.Stiebitz@tu-ilmenau.de}.}}
\date{}	
\maketitle
\begin{abstract}
We prove that
for every $k$ and every $\varepsilon>0$, there exists $g$ such that
every graph with tree-width at most $k$ and odd-girth at least $g$
has circular chromatic number at most $2+\varepsilon$.
\end{abstract}

\section{Introduction}

It has been a challenging problem to prove the existence
of graphs of arbitrary high girth and chromatic
number~\cite{Erd59}. On the other hand, graphs with large
girth that avoid a fixed minor are known to have low chromatic number
(in particular, this applies to graphs embedded on a fixed surface).
More precisely, as Thomassen observed~\cite{Tho88},
a graph that avoids a fixed minor and has large girth is $2$-degenerate,
and hence $3$-colorable. Further, Galluccio, Goddyn and
Hell~\cite{bib-galluccio}
proved the following theorem, which essentially states that
graphs with large girth that avoid a fixed minor are almost bipartite.
\begin{theorem}[Galluccio, Goddyn and Hell, 2001]
\label{thm-ggh}
For every graph $H$ and every $\varepsilon>0$, there exists an integer
$g$ such that the circular chromatic number of
every $H$-minor free graph of girth at least $g$ is at most
$2+\varepsilon$.
\end{theorem}
A natural way to weaken the girth-condition is to
require the graphs to have high odd-girth
(the \emph{odd-girth} is the length of a shortest odd cycle).
However, Young~\cite{You96} constructed $4$-chromatic projective graphs with
arbitrary high odd-girth. Thus, the high odd-girth requirement is not
sufficient to ensure $3$-colorability, even for graphs embedded on a
fixed surface.
Klostermeyer and Zhang~\cite{KlZh00}, though, proved that the circular chromatic
number of every planar graph of sufficiently high odd-girth is
arbitrarily close to $2$. In particular, the same is true for
$K_4$-minor free graphs, i.e. graphs with tree-width at most $2$.
We prove that the conclusion is still true for any class of graphs of
bounded tree-width, which answers a question of Pan and
Zhu~\cite[Question 6.5]{bib-pan} also appearing as Question 8.12
in the survey by Zhu~\cite{bib-zhu01}.

\begin{theorem}\label{thm-main}
For every $k$ and every $\varepsilon>0$, there exists $g$ such that
every graph with tree-width at most $k$ and odd-girth at least $g$
has circular chromatic number at most $2+\varepsilon$.
\end{theorem}

Motivated by tree-width duality, Ne{\v s}et{\v r}il
and Zhu~\cite{bib-nesetril} proved the following theorem.
\begin{theorem}[Ne\v set\v ril and Zhu, 1996]\label{thm-twd}
For every $k$ and every $\varepsilon>0$, there exists $g$ such that
every graph $G$ with tree-width at most $k$ and homomorphic to a graph $H$
with girth at least $g$ has circular chromatic number at most
$2+\varepsilon$.
\end{theorem}

To see that Theorem~\ref{thm-main} implies Theorem~\ref{thm-twd},
observe that if $G$ has an odd cycle of length $g$, then
$H$ has an odd cycle of length at most $g$.

\section{Notation}

A \emph{$(p,q)$-coloring} of a graph is a coloring of the vertices
with colors from the set $\{0,\ldots,p-1\}$ such that the colors of
any two adjacent vertices $u$ and $v$ satisfy $q\le |c(u)-c(v)|\le p-q$.
The \emph{circular chromatic number $\chi_c(G)$} of a graph $G$
is the infimum (and it can be shown to be the minimum) of the ratios $p/q$ such
that $G$ has a $(p,q)$-coloring. For every finite graph $G$,
it holds that $\chi(G)=\lceil\chi_c(G)\rceil$ and
there is $(p,q)$-coloring of $G$ for every $p$ and $q$
with $p/q\ge\chi_c(G)$. In particular, the circular
chromatic number of $G$ is at most $2+1/k$
if and only if $G$ is homomorphic to a cycle of length $2k+1$.
The reader is referred to the surveys by Zhu~\cite{bib-zhu01,bib-zhu06}
for more information about circular colorings.

A \emph{$p$-precoloring} is a coloring $\varphi$ of a subset $A$ of vertices of
a graph $G$ with colors from $\{0,\ldots,p-1\}$,
and its \emph{extension} is a coloring of
the whole graph $G$ that coincides with $\varphi$ on $A$.
The following lemma can be seen as a corollary of a theorem
of Albertson and West~\cite[Theorem 1]{AlWe06}, and it
is the only tool we use from this area.

\begin{lemma}
\label{lm-extend}
For every $p$ and $q$ with $2<p/q$, there exists $d$ such that any $p$-precoloring
of vertices with mutual distances at least $d$ of a bipartite graph $H$
extends to a $(p,q)$-coloring of $H$.
\end{lemma}

A \emph{$k$-tree} is a graph obtained from a complete graph of order $k+1$
by adding vertices of degree $k$ whose neighborhood is a clique.
The \emph{tree-width} of a graph $G$ is the smallest $k$ such that
$G$ is a subgraph of a $k$-tree.
Graphs with tree-width at most $k$ are also called \emph{partial $k$-trees}.

A \emph{rooted partial $k$-tree} is a partial $k$-tree $G$
with $k+1$ distinguished vertices $v_1,\ldots,v_{k+1}$ such that
there exists a $k$-tree $G'$ that is a supergraph of $G$ and
the vertices $v_1,\ldots,v_{k+1}$ form a clique in $G'$.
We also say that the partial $k$-tree is \emph{rooted} at $v_1,\ldots,v_{k+1}$.
If $G$ is a partial $k$-tree rooted at $v_1,\ldots,v_{k+1}$ and
$G'$ is a partial $k$-tree rooted at $v'_1,\ldots,v'_{k+1}$, then
the graph $G\oplus G'$ obtained by identifying $v_i$ and $v'_i$ is
again a rooted partial $k$-tree
(identify the cliques in the corresponding $k$-trees).

Fix $p$ and $q$. If $G$ is a rooted partial $k$-tree,
then $\F(G)$ is the set of all $p$-precolorings of the $k+1$ distinguished
vertices of $G$ that can be extended to a $(p,q)$-coloring of $G$.

The next lemma is a standard application of results in the area
of graphs of bounded tree-width~\cite{RoSe86}.

\begin{lemma}
\label{lm-small}
Let $k$ and $N$ be positive integers such that $N\ge k+1$.
If $G$ is a partial $k$-tree with at least $3N$ vertices,
then there exist partial rooted $k$-trees $G_1$ and $G_2$ such that
$G$ is isomorphic to $G_1\oplus G_2$ and $G_1$ has at least $N+1$ and
at most $2N$ vertices.
\end{lemma}

If $G$ is a partial $k$-tree rooted at $v_1,\ldots,v_{k+1}$,
then its \emph{type} is a $(k+1)\times (k+1)$ matrix $M$ such that
$M_{ij}$ is the length of the shortest path between the vertices
$v_i$ and $v_j$. If there is no such path, $M_{ij}$ is equal to $\infty$.
Any matrix $M$ that is a type of a partial rooted $k$-tree satisfies
the triangle inequality (setting $\infty+x=\infty$ for any $x$). 
A symmetric matrix $M$ whose entries are non-negative
integers and $\infty$ (and zeroes only on the main diagonal) that
satisfies the triangle inequality is a \emph{type}.
A type is \emph{bipartite} if $M_{ij}+M_{jk}+M_{ik}\equiv0\;\mod\; 2$
for any three finite entries $M_{ij}$, $M_{jk}$ and $M_{ik}$.
Two bipartite types $M$ and $M'$ are \emph{compatible} if $M_{ij}$ and $M'_{ij}$
have the same parity whenever both of them are finite.
We define a binary relation on bipartite types as follows:
$M\lm M'$ if and only if $M$ and $M'$ are compatible and
$M_{ij}\leq M'_{ij}$ for every $i$ and $j$.
Note that the relation $\lm$ is a partial order.

We finish this section with the following lemma.
Its straightforward proof is included to help us in
familiarizing with the just introduced notation.

\begin{lemma}
\label{lm-type-glue}
Let $G^1$ and $G^2$ be two bipartite rooted partial $k$-trees
with types $M^1$ and $M^2$
such that there exists a bipartite type $M^0$ with $M^0\lm M^1$ and $M^0\lm M^2$.
Then the types $M^1$ and $M^2$ are compatible, $G^1\oplus G^2$ is a bipartite
rooted partial $k$-tree and its type $M$ satisfies $M^0\lm M$.
\end{lemma}

\begin{proof}
The types $M^1$ and $M^2$ are compatible: if both $M^1_{ij}$ and
$M_{ij}^{2}$ are finite, then $M_{ij}^{0}$ is finite and has the same
parity as $M^1_{ij}$ and $M_{ij}^{2}$. 
Hence, the entries $M^1_{ij}$ and $M^2_{ij}$ have the same parity.

Let $M$ be the type of $G^1\oplus G^2$.
Note that it does not hold in general that $M_{ij}=\min\{M^1_{ij},M^2_{ij}\}$.
We show that $M^0\lm M$ which will also imply that $G^1\oplus G^2$
is bipartite since $M^0$ is a bipartite type.
Consider a shortest path $P$ between two distinguished
vertices $v_i$ and $v_{i'}$ and split $P$ into paths
$P_1,\ldots,P_\ell$ delimited by distinguished vertices on $P$.
Note that $\ell\le k$ since $P$ is a path. Let $j_0=i$ and
let $j_i$ be the index of the end-vertex of $P_i$ for
$i\in\{1,\ldots,\ell\}$.
In particular, $j_\ell=i'$. Each of the paths $P_1,\ldots,P_\ell$
is fully contained in $G^1$ or in $G^2$ (possibly in both if
it is a single edge).
Since $M^0\lm M^1$ and $M^0\lm M^2$, the length of $P_i$
is at least $M^0_{j_{i-1}j_i}$, and it has the same parity as $M^0_{j_{i-1}j_i}$.
Since $M^0$ is a bipartite type (among others, it satisfies the triangle
inequality), the length of $P$, which is $M_{ii'}$, has the same parity as
$M^0_{j_0j_\ell}=M^0_{ii'}$ and is at least $M^0_{ii'}$.
This implies that $M^0\lm M$.
\end{proof}

\section{The Main Lemma}

In this section, we prove a lemma which forms the core of our argument. 
To this end, we first prove another lemma that asserts that for every $k$,
$p$ and $q$, the set of types of all bipartite rooted partial $k$-trees
forbidding a fixed set of $p$-precolorings from extending (and maybe some
other precolorings, too) has always a maximal element.
We formulate the lemma slightly differently to facilitate its application.

\begin{lemma}
\label{lm-mainM}
For every $k$, $p$ and $q$, there exists a finite number
of (bipartite) types $M^1,\ldots,M^m$ such that for any bipartite
rooted partial $k$-tree $G$ with type $M$, there exists a bipartite
rooted partial $k$-tree $G'$ with type $M^i$ for some
$i\in\{1,\ldots,m\}$ such that $\F(G')\subseteq\F(G)$ and $M\lm M^i$.
\end{lemma}

\begin{proof}
Let $d\ge 2$ be the constant from Lemma~\ref{lm-extend} applied for $p$ and $q$.
Let $M^1,\ldots,M^m$ be all bipartite types with entries
from the set $\{1,\ldots,D^{(k+1)^2}\}\cup\{\infty\}$ where $D=4d$.
Thus, $m$ is finite and does not exceed $(D^{(k+1)^2}+1)^{k(k+1)/2}$.

Let $G$ be a bipartite rooted partial $k$-tree with type $M$.
If $M$ is one of the types $M^1,\ldots,M^m$, then there is nothing
to prove (just choose $i$ such that $M=M^i$).
Otherwise, one of its entries is finite and exceeds $D^{(k+1)^2}$.

For $i\in\{1,\ldots,(k+1)^2\}$,
let $J^i$ be the set of all positive
integers between $D^{i-1}$ and $D^i-1$ (inclusively).
Let $i_0$ be the smallest integer such that no entry of $M$
is contained in $J^{i_0}$. Since $M$ has at most $k(k+1)/2$
different entries, such an index $i_0$ exists.
Note that if $i_0=1$, then Lemma~\ref{lm-extend} implies that
$\F(G)$ contains all possible $p$-precolorings, and the sought graph
$G'$ is the bipartite rooted partial $k$-tree composed of $k+1$ isolated vertices,
with the all-$\infty$ type.

Two vertices $v_i$ and $v_j$
at which $G$ is rooted are \emph{close} if $M_{ij}$ is at most $D^{i_0-1}$.
The relation $\approx$ of being close is an equivalence
relation on $v_1,\ldots,v_{k+1}$. Indeed, it is reflexive and
symmetric by the definition, and we show now that it is transitive.
Suppose that $M_{ij}$ and $M_{jk}$ are both at most $D^{i_0-1}$. Then,
the distance between $v_i$ and $v_k$ is at most
$M_{ij}+M_{jk}\le2D^{i_0-1}-2\le D^{i_0}-1$ since $D\ge2$. Consequently,
by the choice of $i_0$, the distance between $v_i$ and $v_k$ is at
most $D^{i_0-1}-1$ and thus $v_i\approx v_k$.

Let $C_1,\ldots,C_{\ell}$ be the equivalence classes of the relation $\approx$.
Note that $C_1,\ldots,C_{\ell}$ is a finer partition than that
given by the equivalence relation of being connected.

Since $G$ is bipartite, we can partition its vertices into
two color classes, say red and blue.
For every $i\in\{1,\ldots,\ell\}$, contract the closed neighborhood
of a vertex $v$ if $v$ is a blue vertex and its distance
from any vertex of $C_i$ is at least $D^{i_0-1}$ and keep doing so
as long as such a vertex exists. Observe that the resulting graph
is uniquely defined. After discarding the components that do not contain
the vertices of $C_i$, we obtain a bipartite partial $k$-tree $G_i$ rooted
at the vertices of $C_i$:
it is bipartite as we have always contracted closed neighborhoods
of vertices of the same color (blue)
to a single (red) vertex, and its tree-width is at most $k$ since
the tree-width is preserved by contractions.
Moreover, the distance between any two
vertices of $C_i$ has not decreased since any path between them through
any of the newly arising vertices has length
at least $2D^{i_0-1}-2\ge D^{i_0-1}$.

Now, let $G'$ be the bipartite rooted partial $k$-tree obtained
by taking the disjoint union of $G_1,\ldots,G_{\ell}$.
The type $M'$ of $G'$ can be obtained from the type of $G$: 
set $M'_{ij}$ to be $M_{ij}$ if the vertices $v_i$ and $v_j$ are close, and
$\infty$ otherwise. Thus, $M'$ is one of the types
$M^1,\ldots,M^m$ and $M\lm M'$.
It remains to show that $\F(G')\subseteq\F(G)$.

Let $c\in\F(G')$ be a $p$-precoloring that extends to $G'$, and recall that $D\ge4$.
For $i\in\{1,\ldots,\ell\}$,
let $A_i$ be the set of all red vertices at distance at most $D^{i_0-1}$ and
all blue vertices at distance at most $D^{i_0-1}-1$ from $C_i$, and
let $R_i$ be the set of all red vertices at distance $D^{i_0-1}-1$ or
$D^{i_0-1}$ from $C_i$.
Set $B_i=A_i\setminus R_i$ ($B_i$ is the ``interior'' of $A_i$ and $R_i$
its ``boundary'').
The extension of $c$ to $G_i$ naturally defines a coloring of all vertices of $A_i$:
$G_i$ is the subgraph of $G$ induced by $A_i$ with some red vertices of
$R_i$ identified (two vertices of $R_i$ are identified if and only if
they are in the same component of the graph $G-B_i$).

Let $H$ be the following auxiliary graph obtained from $G$: remove the vertices
of $B=B_1\cup\cdots\cup B_{\ell}$ and, for $i\in\{1,\ldots,\ell\}$,
identify every pair of vertices of $R_i$
that are in the same component of $G-B$.
Let $R$ be the set of vertices of $H$
corresponding to some vertices of $R_1\cup\cdots\cup R_{\ell}$. Precolor the vertices
of $R$ with the colors given by the colorings of $G_i$ (note that two vertices of $R_i$
in the same component of $G-B_i$ are also in the same component of $G-B$,
so this is well-defined). The graph $H$ is bipartite as only red vertices have been identified.
The distance between any two precolored vertices is at least $d$:
consider two precolored vertices $r$ and $r'$ at distance at most $d-1$.
Let $i$ and $i'$ be such that $r\in R_i$ and $r'\in R_{i'}$. If $i=i'$,
then $r$ and $r'$ are in the same component of $G-B$ and thus $r=r'$.
If $i\not=i'$ then
by the definition of $R_i$ and $R_{i'}$, the vertex $r$ is in $G$ at distance at most $D^{i_0-1}$
from some vertex $v$ of $C_i$ and $r'$ is at distance at most $D^{i_0-1}$ from
some vertex $v'$ of $C_{i'}$. So, the distance between $v$ and $v'$
is at most $2D^{i_0-1}+d<D^{i_0}-1$. Since $M$ has no entry from $J^{i_0}$,
the vertices $v$ and $v'$ must be close and thus $i=i'$, a
contradiction.

Since the distance between any two precolored vertices is at least $d$, the precoloring
extends to $H$ by Lemma~\ref{lm-extend} and in a natural way it defines a coloring of $G$.
We conclude that
every $p$-precoloring that extends to $G'$ also extends to $G$ and thus $\F(G')\subseteq\F(G)$.
\end{proof}

We now prove our main lemma, which basically states that there is only
a finite number of bipartite rooted partial $k$-trees that can appear
in a minimal non-$(p,q)$-colorable graph with tree-width $k$ and
a given odd girth.

\begin{lemma}
\label{lm-mainG}
For every $k$, $p$ and $q$, there exist a finite number $m$ and
bipartite rooted partial $k$-trees $G^1,\ldots,G^m$ with types
$M^1,\ldots,M^m$ such that for any bipartite rooted partial
$k$-tree $G$ with type $M$ there exists $i$ such that
$\F(G^i)\subseteq \F(G)$ and $M\lm M^i$.
\end{lemma}

\begin{proof}
Let $M^1,\ldots,M^{m}$ be the types from Lemma~\ref{lm-mainM}.
We define the graph $G^i$ as follows:
for every $p$-precoloring $c$ that does not extend to a bipartite partial
rooted $k$-tree with type $M^i$, fix any partial rooted $k$-tree $G^i_c$
with type $M^i$ such that $c$ does not extend to $G^i_c$.
Set $G^i=\bigoplus_{c} G^i_c$, where $c$ runs over all such $p$-precolorings.
If the above sum of partial $k$-trees is non-empty, then the type $M$
of $G^i$ is $M^i$. Indeed,
$M\lm M^i$ by the definition of $G^i$, and Lemma~\ref{lm-type-glue}
implies that $M^i\lm M$.
If all the $p$-precolorings of the $k+1$ vertices in the root extend to
each partial $k$-tree of type $M^i$, then let
$G^i$ be the graph consisting of $k+1$ isolated vertices. This happens
in particular for the all-$\infty$ type.

Let us verify the statement of the lemma.
Let $G$ be a bipartite rooted partial $k$-tree and
let $M$ be the type of $G$. If $\F(G)$ is composed of all $p$-precolorings,
the sought graph $G^i$
is the one composed of $k+1$ isolated vertices. Hence, we assume that $\F(G)$
does not contain all $p$-precolorings, i.e., there are $p$-precolorings that do not extend to $G$.
By Lemma~\ref{lm-mainM}, there exists a bipartite rooted partial $k$-tree
$G'$ with type $M'$ such that $M\lm M'=M^i$ for some $i$ and $\F(G')\subseteq\F(G)$.
For every $p$-precoloring $c$ that does not extend to $G'$
(and there exists at least one such $p$-precoloring $c$),
some graph $G^i_c$ has been glued into $G^i$.
Hence, $\F(G^i)\subseteq\F(G')\subseteq\F(G)$. Since the type of $G^i$ is $M^i$,
the conclusion of the lemma follows.
\end{proof}

\section{Proof of Theorem~\ref{thm-main}}

We are now ready to prove Theorem~\ref{thm-main}, which is recalled below.

\begin{thm2}
For every $k$ and every $\varepsilon>0$, there exists $g$ such that
every graph with tree-width at most $k$ and odd-girth at least $g$
has circular chromatic number at most $2+\varepsilon$.
\end{thm2}

\begin{proof}
Fix $p$ and $q$ such that $2<p/q\le 2+\varepsilon$. Let $G^1,\ldots,G^m$
be the bipartite partial $k$-trees from Lemma~\ref{lm-mainG} applied
for $k$, $p$ and $q$.
Set $N$ to be the largest order of the graphs $G^i$ and set $g$ to be $3N$.
We assert that each partial $k$-tree with odd-girth $g$ has circular chromatic
number at most $p/q$. Assume that this is not the case and
let $G$ be a counterexample with the fewest vertices.

The graph $G$ has at least $3N$ vertices (otherwise, it has no odd cycles and
thus it is bipartite). By Lemma~\ref{lm-small}, $G$ is isomorphic to
$G_1\oplus G_2$, where $G_1$ and $G_2$ are rooted partial $k$-trees and the number of vertices of $G_1$
is between $N+1$ and $2N$. By the choice of $g$, the graph
$G_1$ has no odd cycle and thus
it is a bipartite rooted partial $k$-tree. By Lemma~\ref{lm-mainG}, there exists $i$
such that $\F(G^i)\subseteq\F(G_1)$ and $M_1\lm M^i$
where $M_1$ is the type of $G_1$ and $M^i$ is the type of $G^i$.
Let $G'$ be the partial $k$-tree $G^i\oplus G_2$.

First, $G'$ has fewer vertices than $G$ since the number of vertices of $G^i$
is at most $N$ and the number of vertices of $G_1$ is at least $N+1$.
Second, $G'$ has no $(p,q)$-coloring: if it had a $(p,q)$-coloring, then the
corresponding
$p$-precoloring of the $k+1$ vertices shared by $G^i$ and $G_2$ would extend
to $G_1$ since $\F(G^i)\subseteq\F(G_1)$
and thus $G$ would have a $(p,q)$-coloring, too.
Finally, $G'$ has no odd cycle of length at most $g$: if it had such a cycle,
replace any path between vertices $v_j$ and $v_{j'}$ of the root of $G^i$ with a path
of at most the same length between them in $G_1$ (recall that $M_1\lm M^i$). If such paths
for different pairs of $v_j$ and $v_{j'}$ on the considered odd cycle intersect,
take their symmetric difference. In this way, we obtain an Eulerian subgraph of
$G=G_1\oplus G_2$ with an odd number of edges such that the number of its edges
does not exceed $g$. Consequently, this Eulerian subgraph has an odd cycle of
length at most $g$, which violates the assumption on the odd-girth of $G$.
We conclude that $G'$ is a counterexample with less vertices than
$G$, a contradiction.
\end{proof}

We end by pointing out that the approach used yields an upper bound
of
$3(k+1)\cdot2^{2^{p^{k+1}}((4d)^{(k+1)^2}+1)^{k^2}}$
for the smallest $g$ such that
all graphs with tree-width at most $k$ and odd-girth at least $g$
have circular chromatic number at most $p/q$, whenever $p/q>2$.
More precisely, the value of $N$ cannot exceed
$(k+1)\cdot2^{2^{p^{k+1}}((4d)^{(k+1)^2}+1)^{k^2}}$. To see this,
we consider all pairs $P=(C,M)$ where $C$ is a set of $p$-precolorings
of the root and $M$ is a type such that there is a bipartite rooted
partial $k$-tree of type $M$ to which no coloring of $C$ extends.
Let $n_P$ be the size of a smallest such partial $k$-tree.
We obtain a sequence of at most 
$2^{p^{k+1}}\times\left((4d)^{(k+1)^{2}}+1\right)^{k^2}$
integers. The announced bound follows from the following fact:
if the sequence is sorted in increasing order, then each term is at most twice
the previous one.

Indeed, consider the tree-decomposition of the partial $k$-tree $G_P$
chosen for the pair $P$. If the bag containing the root has
a single child, then we delete a vertex of the root, and set a vertex
in the single child to be part of the root. We obtain a partial $k$-tree
to which some $p$-precolorings of $C$ do not extend. Thus,
$n_P\le1+n_{P'}$ for some pair $P'$ and $n_{P'}<n_P$.
If the bag containing the root has more than one child, then
$G_P$ can be obtained by identifying the roots of two smaller partial
$k$-trees $G$ and $G'$.
By the minimality of $G_P$, the orders of $G$ and $G'$ are $n_{P_1}$ and
$n_{P_2}$ for two pairs $P_1$ and $P_2$ such that $n_{P_i}<n_P$ for $i\in\{1,2\}$.
This yields the stated fact, which in turn implies
the given bound, since the smallest element of the sequence is $k+1$.

\bigskip
\noindent
\textbf{Acknowledgment.} This work was done while the first three
authors were visiting the fourth at Technische Universit\"at Ilmenau.
They thank their host for providing a perfect working environment.

\end{document}